\documentclass{amsart}
\usepackage{amsfonts,amssymb,amscd,amsmath,enumerate,verbatim,calc}
\usepackage[all]{xy}

\newcommand{\CM}{Cohen-Macaulay}

\newcommand{\wrt}{with respect to}

\newcommand{\n}{\mathfrak{n} }
\newcommand{\m}{\mathfrak{m} }

\newcommand{\bP}{\mathbb{\partial}}

\newcommand{\rt}{\rightarrow}

\newcommand{\ov}{\overline}

\newcommand{\K}{\mathbf{K}_{\bullet} }

\newcommand{\bx}{\mathbf{x}}
\newcommand{\ba}{\mathbf{a}}

\newcommand{\image}{\operatorname{image}}

\newcommand{\BigFig}[1]{\parbox{12pt}{\Huge #1}}
\newcommand{\BigZero}{\BigFig{0}}

\newcommand{\hh}{\operatorname{height}}
\newcommand{\ann}{\operatorname{ann}}

\newcommand{\Supp}{\operatorname{Supp}}

\newcommand{\Spec}{\operatorname{Spec}}

\newcommand{\soc}{\operatorname{soc}}
\newcommand{\Hom}{\operatorname{Hom}}
\newcommand{\Ext}{\operatorname{Ext}}
\newcommand{\Ass}{\operatorname{Ass}}
\newcommand{\Min}{\operatorname{Min}}

\theoremstyle{plain}

\newtheorem{thm}{Theorem}

\newtheorem{lem}[thm]{Lemma}
\newtheorem{theorem}{Theorem}[section]

\newtheorem{lemma}[theorem]{Lemma}
\newtheorem{proposition}[theorem]{Proposition}

\theoremstyle{definition}

\newtheorem{s}[theorem]{}
\newtheorem{remark}[theorem]{Remark}

\theoremstyle{remark}

\numberwithin{equation}{theorem}

\begin{document}

\title{Injective modules over some rings of Differential operators}
 \author{Tony~J.~Puthenpurakal}
\date{\today}
\address{Department of Mathematics, Indian Institute of Technology Bombay, Powai, Mumbai 400 076}

\email{tputhen@math.iitb.ac.in}


\begin{abstract}
Let $R$ be a regular domain containing a field $K$ of characteristic zero and let $D$ be the ring of $K$-linear differential operators on $R$. Let $E$ be an injective left $D$-module. We ask the question, when is $E$ injective as a
$R$-module? We show that this is indeed the case when $R = K[X_1,\ldots,X_n]$ or $R = K[[X_1,\ldots,X_n]]$ or $R = \mathbb{C}\{z_1,\ldots,z_n\}$. We also give an application of our result to local cohomology.
\end{abstract}

\maketitle

\section*{Introduction}
The motivation for this paper comes from a problem in local cohomology which we now describe. Let $K$ be a field of characteristic zero, $R = K[X_1,\ldots,X_n]$ and let $I$ be an ideal in $R$. Let $A_n(K) = K<X_1,\ldots,X_n, \partial_1, \ldots, \partial_n>$  be the $n^{th}$ Weyl algebra over $K$. By a result due to Lyubeznik, see \cite{Ly}, the local cohomology modules $H^i_I(R)$ are finitely generated $A_n(K)$-modules for each $i \geq 0$. If $J$ is another
ideal in $R$ then we have a Mayer-Vietoris sequence
\[
\cdots \rt H^{i}_{I+J}(R) \xrightarrow{\rho^i} H^i_I(R) \oplus H^i_J(R) \xrightarrow{\pi^i} H^i_{I\cap J}(R) \xrightarrow{\delta^i} H^{i+1}_{I+J}(R) \rt \cdots
\]
It can be easily seen that for all $i \geq 0$; $\rho^i,\pi^i$ are in-fact $A_n(K)$-linear; for instance see
\cite[1.5]{P}.
 A natural question is whether
$\delta^i$ is $A_n(K)$-linear for all $i \geq 0$? In this paper we show this is so; see Proposition \ref{app}. The crucial ingredient is the following: Let $E$ be an injective left $A_n(K)$ module and since $R$ is a subring of $A_n(K)$ consider $E$
as a $R$-module. Then we prove that $E$ is an injective $R$-module. More generally we prove the following result:
\begin{thm}\label{main}
Let $S$ be a ring containing a regular commutative ring $R$ as a subring. Assume that $S$ considered as a right $R$-module is projective. Let $E$ be a left $S$-module which is injective as a $S$-module. Then $E$ is an injective $R$-module.
\end{thm}
Our motivation of $S$ in the theorem above is ring's of differential operators. Note that when $R = K[X_1,\ldots,X_n]$ or $R = K[[X_1,\ldots,X_n]]$ or $R = \mathbb{C}\{z_1,\ldots,z_n\}$ and $S$ is the ring of differential operators on $R$ then $S$ satisfies the hypotheses of Theorem \ref{main}, see section 1. After we proved the result, we observed that there are
many other examples of rings which satisfy the hypotheses of Theorem \ref{main}. In section 1 we have listed many examples  of $S$ which satisfy the hypotheses in Theorem 1.

The main technical tool used to prove Theorem \ref{main} is the following result:
\begin{lem}\label{chief}
Let $R$ be a regular commutative ring and let $E$ be a $R$-module. If $\Ext_{R}^{1}(R/J,E) = 0$ for every ideal $J$
 generated by a regular sequence, then $E$ is an injective $R$-module.
\end{lem}

Matlis theory of injective modules over left Noetherian rings, in particular over commutative Noetherian rings is a
 very basic tool.
We analyze the structure of injective $S$-modules considered as a $R$-module. For this we first assume that $S$ is a left Noetherian ring. 
 We make a further assumption:

$(*)$  Given an ideal $I$ in $R$ and $s \in S$, there exists $r \geq 1$ ($r$ depending on $s$) such that $I^rs \subseteq SI$.

This hypothesis is satisfied when $S$ is the ring of differential operators. It is also trivially satisfied when $S$ is commutative. We show
\begin{thm}\label{ind-I}
(with assumptions as above). Let $M$ be an $S$-module with $E_S(M)$ an indecomposable injective $S$-module. Then
$\Ass_R E_S(M)$ has a unique maximal element $P$. Furthermore $P \in \Ass_R M$.
\end{thm}

A particulary interesting case is when $S$ is commutative and finitely generated as a $R$-module. Note that we are also assuming that $S$ is projective as a $R$-module. This case is equivalent to assuming $S$ is Cohen-Macaulay as a ring. 
We prove
\begin{thm}\label{CM-I}
(with assumptions as above). Let $\m$ be a maximal ideal of $S$. Set $\n = \m \cap R$.
Let $\n S = Q_1\cap Q_2\cap\cdots \cap Q_c$ be a minimal primary decomposition of $\n S$ where $Q_1$ is $\m$-primary.
Then
\[
E_S\left(\frac{S}{\m}\right) = E_R\left(\frac{R}{\n}\right)^{\ell_R(S/Q_1)}.
\]
\end{thm}

We now describe in brief the contents of the paper. In section 1 we give many examples  of rings $S$ which satisfy the hypotheses for Theorem \ref{main}. In section 2 we prove Lemma 2. In section 3 we prove Theorem 1. 
In section 3 we also give application of our result in the theory of local cohomology; as described in the begininig of this section.
In section 4 we prove Theorem 3. In section 5 we prove Theorem 4. 
\section{Examples}
In this section we give many examples  of rings $S$ which satisfy the hypotheses for Theorem \ref{main}.

\textbf{Example 1:}  Let $K$ be a field of characteristic zero, $R = K[X_1,\ldots,X_n]$ and  let $S = A_n(K) = K<X_1,\ldots,X_n, \partial_1, \ldots, \partial_n>$  be the $n^{th}$ Weyl algebra over $K$. By \cite[p.\ 3, 1.2]{B}
it follows that every $s \in S$ has a unique expression
\[
s = \sum_{\alpha \in \mathbb{N}^n} \phi_{\alpha}(X)\bP^\alpha \quad \text{with}\ \phi_{\alpha}(X) \in R.
\]
In other words $S$ is free as a left $R$-module. The fact that $S$ is also free as a right $R$-module is also known. However due to lack of a reference we give a proof here.

(a) Using the defining relations of the Weyl algebra it is clear that any $s \in S$ can be written as
\begin{equation*}
s = \sum_{\alpha \in \mathbb{N}^n} \bP^\alpha \psi_{\alpha}(X) \quad \text{with}\ \psi_{\alpha}(X) \in R. \tag{$\dagger$}
\end{equation*}
(b) We prove that the above expression is unique. For this we need the following Lemma which is easy to prove
\begin{lemma}\label{temp}
Let $A$ be a ring containing a field $K$ and let $g$ be a $K$-linear derivation on $R$. Then in the ring $D_K(A)$
of $K$-linear differential operators, for any $a\ in A$ we have

$\displaystyle{\quad \ \ \ \ \ \ \ \ \ \ \ \ \ \ \ \ \ \ \ g^m a = \sum_{i = 0}^{m} \binom{m}{i}g^{m-i}(a)g^i.}$
\qed
\end{lemma}
So let $s = \sum_{\alpha \in \mathbb{N}^n} \bP^\alpha \psi_{\alpha}(X) = 0$. If possible assume that some $\psi_{\alpha}(X) \neq 0$. Let $\alpha^* \in \mathbb{N}^n$ be such that $\psi_{\alpha^*}(X) \neq 0$ and $\psi_{\alpha}(X) = 0 $ for all $\alpha$ with $|\alpha| > |\alpha_*|$. Clearly $\alpha^* \neq 0$. Also using \ref{temp}
we get that
\[
s = \sum_{\alpha^*}  \psi_{\alpha^*}(X)\bP^{\alpha^*} + \sum_{|\alpha| < |\alpha_*|}c_\alpha  \bP^\alpha.
\]
Since $s = 0$ it follows that  $\psi_{\alpha^*}(X) = 0$ for all $\alpha^*$, a contradiction. Thus the expression in $(\dagger)$ is unique and so $S$ is free as a right $R$-module.

\textbf{Example 2:} Let $K$ be a field of characteristic zero, $R = K[[X_1,\ldots,X_n]]$ and  let $S$ = ring of $K$-linear differential operators on $R$. Then $S = R[\partial_1,\ldots,\partial_n]$.
Every $s \in S$ has a unique expression
\[
s = \sum_{\alpha \in \mathbb{N}^n} \phi_{\alpha}(X)\bP^\alpha \quad \text{with}\ \phi_{\alpha}(X) \in R.
\]
In other words $S$ is free as a left $R$-module. By an argument similar to that in Example 1 we can show that
there every $s \in S$ has a unique expression
\[
s = \sum_{\alpha \in \mathbb{N}^n} \bP^\alpha \psi_{\alpha}(X) \quad \text{with}\ \psi_{\alpha}(X) \in R.
\]
So $S$ is  free as a right $R$-module.

\textbf{Example 3:} $R = \mathbb{C}\{z_1,\ldots,z_n\}$ the ring of convergent power series with complex coefficients. Let $S$ = ring of $\mathbb{C}$-linear differential operators on $R$. Then $S = R[\partial_1,\ldots,\partial_n]$. As in Example 2 we can prove that $S$ is  free as a right $R$-module.

\textbf{Example 4:} $R$ any commutative regular ring and $S = R[X_1,\ldots,X_n]$. Then clearly $S$ is free as a right $R$-module.

\textbf{Example 5:} $R$ any commutative regular ring and $S = M_n(R)$; the ring of $n\times n$ matrices over $R$. Clearly $S \cong R^{n^2}$ as a right $R$-module. The ring $R$ can be considered as a subring of $R$ via scalar matrices. Note that $R \subseteq Z(S)$; the center of $S$.

\textbf{Example 6:}(A non-Noetherian example) Let $R = \mathbb{Z}$ and let $S \subseteq \mathbb{C}$ be the subring of all algebraic integers. Then $S$ is torsion-free $R$-module, so it is free.

\textbf{Example 7:}(Differential polynomial rings) $R$ any commutative regular ring and let $\delta \colon R \rt R$
be a derivation on $R$. Consider the differential polynomial ring $S = R[X; \delta]$. An element $s$ in $S$ has a unique expression $s = \sum_{i = 0}^{m} a_i X^i$. So $S$ is a free left $R$-module. Note that
$Xa = aX + \delta a$. We can prove that
\[
X^na = \sum_{i =0}^{n} \binom{n}{i} \delta^i(a) X^{n-i}.
\]
So by an argument similar to in Example 1 we can show that each element $s$ in $S$ has a unique expression $s = \sum_{i = 0}^{m}X^ib_i$.  It follows that $S$ is free as a right $R$-module.

\textbf{Example 8:}(Group rings) $R$ any commutative regular ring and let $G$ be a group. Consider the group ring $S =
R[G] = \bigoplus_{\sigma \in G} R\sigma$. Notice $r\sigma = \sigma r$ for any $r \in R$ and $\sigma \in G$. It follows that $R[G] = \bigoplus_{\sigma \in G} \sigma R$. It follows that $S$ is free as a right $R$-module. Also note that
$R \subseteq Z(S)$.

\textbf{Example 9:}  Let $K$ be a field of characteristic zero, $A = K[X_1,\ldots,X_n]$ and  let $f \in A$ be a non-constant polynomial. Set $R = A_f$ and let $S$ be the ring of $K$-linear differential operators on $R$. Then
$S = R[\partial_1,\ldots,\partial_n]$.  As in Example 1 we can prove that $S$ is a free as a right $R$-module.

\textbf{Example 10:} Let $S$ be a \CM \ affine algebra over  an infinite field $K$ and let $R$ be a Noether normalization of $S$. Then $R$ is a regular subring of $S$ and since $S$ is \CM, $S$ is projective as a $R$-module.
As all projectives over $R$ are free we get that $S$ is free as a $R$-module.

In Examples 1-10, $S$ was a free right $R$-module. We now give two examples when $S$ is only projective as a $R$-module.

\textbf{Example 11:} Let $R$ be a Dedekind  domain with quotient field $K$ and let $L$ be a finite field extension of $K$. Set $S$ to be the integral closure of $R$ in $L$. Then $S$ also a Dedekind ring. In general $S$ is a projective $R$-module and not-necessarily free. In fact if $K$ is a number-field with $R$ not a P.I.D then there always exists a finite extension $L$ with $S$ not-free as a $R$-module; see \cite{Mann}. For a specific example see \cite{MS}. The author thanks Sudhir Ghorpade for this specific example.

\textbf{Example 12:} Let $R$ be a regular domain having a projective module $N$ such that $S = \Hom_R(N,N)$ is not free
as a $R$-module. Clearly $S$ is projective as a $R$-module. Also note  that $R$ can be considered as a subring of $S$ via the map $i \colon R \rt S$ where $i(r)$ is the multiplication map. Clearly $R$ is in the center of $S$.
The following specific example was constructed by Manoj Keshari. Recall a projective module $P$ is said to be cancellative if $P\oplus R^n \cong Q\oplus R^n$ implies $P \cong Q$. Let $A $ be the homogeneous localization
$\mathbb{R}[X,Y,Z]_{(X^2 + Y^2 + Z^2)}$ and let $R = A_0$. Then $R$ is a smooth affine surface. The projective module
$K_R \oplus R$ is not cancellative (here $K_R$ denotes the canonical module of $R$); see \cite[Example 3.1]{Bhat}. Since
$K_R \oplus R$ is not cancellative then there exists a projective module $P$ with $K_R \oplus R \oplus R^r  \cong
P \oplus R^r$, but $ P \ncong K_R \oplus R$. It can be shown that $\Hom_R(P,P)$ is not free as a $R$-module.
\section{Proof of Lemma \ref{chief}}
In this section we give a proof Lemma \ref{chief}. We need the following result.
\begin{proposition}\label{cp}
Let $R$ be a Cohen-Macaulay commutative ring and let $P$ be a prime ideal in $R$ with   $\hh P \geq g \geq 1$. Then there exists an $R$-sequence $x_1,\ldots,x_g \in P$ such that $\frac{x_1}{1},\ldots,\frac{x_g}{1}$ is part of minimal generators of $PR_P$ in the local ring $R_P$.
\end{proposition}
\begin{proof}
We prove the result by induction on $g$. We first consider the case when  $g = 1$. Consider $P^{(2)} = P^2R_P \cap R$.
Note that $P^{(2)} \subseteq P$. If $P \subseteq P^{(2)}$ then $P = P^{(2)}$. So $PR_P = P^2R_P$. By Nakayama Lemma $PR_P =0$. Thus $\hh P = \dim R_P = 0$ a contradiction since $\hh P \geq g = 1$. Let $Q_1,\ldots, Q_s$ be minimal primes of $R$. Since $R$ is \CM \ they are also all the associate primes of $R$. So $P \nsubseteq Q_i$ for $i = 1,\ldots,s$.
Also we have shown that $P \nsubseteq P^{(2)}$. So by prime avoidance, \cite[Lemma 3.3]{E} there exists
\[
x_1 \in P \setminus \left( P^{(2)} \cup \left( \cup_{i = 1}^{s} Q_i  \right) \right).
\]
Clearly $x_1$ is $R$-regular. Also $\frac{x_1}{1} \notin P^2R_P$ since $x_1 \notin P^{(2)}$. Thus $\frac{x_1}{1}$
is part of minimal system of generators of $PR_P$.

We assume the result for $g = i$ and prove the result for $g = i + 1$. Let $P$ be a prime ideal with $\hh P \geq g = i+ 1$. By induction hypotheses there exists $x_1,\ldots,x_i \in P$ such that $x_1,\ldots,x_i$ is a $R$-regular sequence and $\frac{x_1}{1},\ldots,\frac{x_i}{1}$ is part of minimal generators of $PR_P$. Clearly $\hh (x_1,\ldots,x_i) = i$.
Let
\[
\{ Q_1,\ldots, Q_r \} = \Min \frac{R}{(x_1,\ldots,x_i)} = \Ass \frac{R}{(x_1,\ldots,x_i)}.
\]
The last equality holds since $R$ is \CM. By assumption
$\hh P \geq i +1$. So $P \nsubseteq Q_j$ for all $j = 1,\ldots, r$. Set
\[
J = <\frac{x_1}{1},\ldots,\frac{x_i}{1}>R_P + P^2R_P
\]
and $L = J \cap R$. Note that $L \subseteq P$. We claim that $P \nsubseteq L$. Otherwise $P = L$.
So $PR_P = J$. Therefore
\[
\frac{PR_P}{P^2R_P} = <\ov{\frac{x_1}{1}},\ldots,\ov{\frac{x_i}{1}}>.
\]
This implies $\dim R_P \leq i$ a contradiction since   $\dim R_P = \hh P \geq i+1$. Thus $P \nsubseteq L$.
By prime-avoidance there exists
\[
x_{i+1} \in P \setminus \left( L \cup \left( \cup_{i = 1}^{r} Q_i  \right) \right).
\]
Clearly $x_1,\ldots,x_{i+1}$ is a $R$-regular sequence and $\frac{x_1}{1},\ldots,\frac{x_{i+1}}{1}$ is part of minimal generators of $PR_P$. Thus by induction our result is true.
\end{proof}

We now give
\begin{proof}[Proof of Lemma \ref{chief}]
Let $P$ be a prime ideal in $R$. Set $\kappa(P) = R_P/PR_P$.
We first show
\begin{equation*}
\Ext^{1}_{R_P}\left( \kappa(P), E_P \right) = 0. \tag{$\dagger$}
\end{equation*}
First consider the case when $P$ is a minimal prime of $R$. Since $R$ is a regular ring, $R_P$ is a field. So $PR_P = 0$. Thus $\kappa(P) = R_P$. Clearly
\[
\Ext^{1}_{R_P}\left( R_P , E_P \right) = 0.
\]
Suppose $\hh P = g \geq 1$. Then by Proposition \ref{cp} there exists $x_1,\ldots,x_g \in P$ such that $x_1,\ldots,x_g$ is a $R$-regular sequence and $\frac{x_1}{1},\ldots,\frac{x_g}{1}$ is part of minimal generators of $PR_P$. Set $J = (x_1,\ldots,x_g)$. Since $R$ is regular, $R_P$ is a regular local ring. It follows that $PR_P$ is minimally generated by $g$ elements. Thus $JR_P = PR_P$. By our hypothesis $\Ext^{1}_{R}(R/J, E) = 0$. Localizing we get
\[
\Ext^{1}_{R_P}\left( \kappa(P), E_P \right) = \Ext^{1}_{R_P}(R_P/JR_P, E_P) = 0.
\]
Thus we have shown $(\dagger)$ for every prime ideal $P$ of $R$.

Let $0 \rt E \rt I^0 \rt I^1 \rt \cdots $ be a minimal injective resolution of $E$. We have, see \cite[18.7]{M},
\[
I^1 = \bigoplus_{P \in \Spec(R)} E(R/P)^{\mu_1(P,M)} \quad \text{where} \ \mu_1(P,M) = \dim_{\kappa(P)}\Ext^{1}_{R_P}\left( \kappa(P), E_P \right).
\]
So $I^1 = 0$. It follows that $E \cong I^0$. Thus $E$ is an injective $R$-module.
\end{proof}

\section{Proof of Theorem \ref{main}}
In this section we give a proof of Theorem \ref{main}. We also give application of our result to our problem local cohomology  We need the following lemma
\begin{lemma}\label{reg}
Let $A$ be a commutative ring and let $P$ be a projective $A$-module. Let $x_1,\ldots,x_n$ be a $A$-regular sequence. Then $x_1,\ldots,x_n$ is a $P$-regular sequence.
\end{lemma}
\begin{proof}
Let $Q$ be a $A$-module with $P\oplus Q = F$, free. Clearly $x_1,\ldots,x_n$ is a $F$-regular sequence. It follows that
$x_1,\ldots,x_n$ is a weak $P$-regular sequence. It remains to prove that $(\bx)P \neq P$. Since $A \neq (\bx)$ it follows that there exists a maximal ideal $\m$ of $A$ containing $(\bx)$. Since $P_\m$ is a free $A_\m$-module, see \cite[Theorem 2.5]{M}, we have
that $(\bx)P_\m \neq P_\m$. So $(\bx)P \neq P$.
\end{proof}

\begin{remark}\label{right}
Let $A$ be a commutative ring.
Suppose $M$ is a right $A$-module and let $\bx = x_1,\ldots,x_n$ be a  sequence in $A$. Then for the Koszul complex
$\K(\bx;M)$ we consider elements in $K(\bx;M)_i$ as "row's" and not as columns as is the practice when $M$ is a left $R$-module. This is natural since $M$ is a right $A$-module. Also note that the maps in $K(\bx;M)$ are transposes of the usual maps when $M$ is a left $R$-module.
\end{remark}
\begin{proof}[Proof of Theorem \ref{main}]
By Lemma \ref{chief} it suffices to show that $\Ext_{R}^{1}(R/J,E) = 0$ for every ideal $J$
 generated by regular sequence.

 Let $J = (a_1,\ldots,a_g)$ where $a_1,\ldots,a_g$ is a $R$-regular sequence and let $\phi \colon J \rt E$ be a $R$-linear map. We want to prove that  there exists $R$-linear map $\widetilde{\phi} \colon R \rt E$ with $\widetilde{\phi}_J = \phi$.

 Set $L = SJ$ the left ideal in $S$ generated by $J$. Define
 \begin{align*}
 \psi \colon L &\rt E \\
 x \in L; \text{if} \ x &= s_1a_1 + \cdots + s_ga_g \ \text{then}\\
 \psi(x) &= s_1\phi(a_1) + \cdots + s_g\phi(a_g).
 \end{align*}

\textit{Sub-Lemma:}  $\psi$ is well-defined, i.e.,
\[
 \text{if} \ x = s_1a_1 + \cdots + s_ga_g = t_1a_1 + \cdots + t_ga_g
\]
\[
\text{then} \ s_1\phi(a_1) + \cdots + s_g\phi(a_g) = t_1\phi(a_1) + \cdots + t_g\phi(a_g).
\]
We first assume Sub-Lemma. Note that $\psi$ is $S$-linear. For if $x =  s_1a_1 + \cdots + s_ga_g$
then for $t \in S$ we have $tx = ts_1 a_1 + \cdots + t s_g a_g$. So
\[
\psi(tx) = ts_1\phi(a_1) + \cdots + ts_g\phi(a_g) = t \psi(x).
\]
Since $E$ is an injective $S$-module, there exists an $S$-linear map $\widetilde{\psi} \colon S \rt E$ with $\widetilde{\psi}_L = \psi$.

Note that $\psi_J = \phi$; for if $a \in J$ then $a = r_1 a_1 + \cdots + r_g a_g$ for some $r_1,\ldots,r_g \in R$.
So
\[
\psi(a) = r_1\phi(a_1) + \cdots + r_g\phi(a_g) = \phi(a), \quad \text{since} \ \phi \ \text{is $R$-linear}.
\]
Let $i \colon R \rt S$ be the inclusion map. Set $\widetilde{\phi} = \widetilde{\psi} \circ i$. Note that
$\widetilde{\phi} \colon R \rt E$ is $R$-linear and clearly $\widetilde{\phi}_J = \phi$. Thus it remains to prove the
sub-lemma.

\textit{Proof of Sub-Lemma:} By Lemma \ref{reg} we have that $a_1,\ldots,a_g$ is a $S$-regular sequence; here $S$ is considered as a right $R$-module.  For the convenience of the reader we first give the proof when $g = 1, 2$ and then give a
general argument.

First consider the case when $g = 1$. So $J = (a_1)$ and $L = SJ = Sa_1$. Let $x \in L$. If $x = s_1a_1 = t_1a_1$, then as $a_1$ is $S$-regular we have $s_1 = t_1$. So $s_1\phi(a_1) = t_1 \phi(a_1)$.

Next consider the case when $g = 2$. So $J = (a_1,a_2)$ and $L = SJ$.  Let $x \in L$. If
$$x = s_1a_1 + s_2 a_2 = t_1a_1 + t_2a_2; $$
then as $a_1, a_2$ is a $S$-regular sequence there exists $c \in S$ with  $s_1 = t_1 - ca_2$ and $s_2 = t_2 + ca_2$. Thus we have
\[
s_1\phi(a_1) + s_2 \phi(a_2) = t_1\phi(a_1) + t_2 \phi(a_2) - c\left(a_2\phi(a_1) - a_1\phi(a_2) \right).
\]
Since $\phi$ is $R$-linear we have that $a_2\phi(a_1) - a_1\phi(a_2) = 0$. Thus the result follows when $g = 2$.

For the general argument, consider the Koszul complex $\K(\ba;S)$. Before proceeding further please read Remark \ref{right}. Let $J = (a_1,\ldots,a_g)$, with $g \geq 2$, and let $L = SJ$.
Let the Koszul maps be given as
\[
\cdots S^{\binom{g}{2}} \xrightarrow{\psi_g} S^g \xrightarrow{\phi_g} S \rt 0
\]
Let $x \in L$. If
\[
x = s_1a_1 + s_2 a_2 + \cdots s_g a_g = t_1a_1 + t_2a_2 + \cdots t_g a_g;
\]
then $u = [s_1 -t_1, s_2 -t_2,\cdots,s_g -t_g] \in \ker \phi_g$. Since $\ba$ is a $S$-regular sequence;  $\K(\ba;S)$ is acyclic. So $u \in \image \psi_g$. Say $u = c\psi_g$ for some $c \in S^{\binom{g}{2}}$. Let
\[
w_g = [\phi(a_1),\ldots,\phi(a_g)]^{tr} \quad \text{where "tr" denotes transpose}.
\]
Then note that
\[
\sum_{i = 1}^{g}(s_i -t_i)\phi(a_i) = uw_g = c\psi_g w_g.
\]
Thus it is sufficient to prove $\psi_g w_g = 0$. This we prove by induction on $g$ where $g \geq 2$.
We first consider the case when $g = 2$. Notice that
\[
\psi_2 = [-a_2,a_1].
\]
So $\psi_2w_2 = -a_2\phi(a_1) + a_1 \phi(a_2) = 0$, since $\phi$ is $R$-linear.

We assume the result when $g = r-1$ and prove it when $g = r$; (here $r \geq 3$).
Notice that
\[
\psi_{r} = \left(
\begin{matrix}
-a_2&a_1&0 &0\ldots 0  &0\\
-a_3&0 &a_1 &0 \ldots 0 &0\\
\cdots \\
-a_r&0 &0 &0 \ldots 0 &a_1 \\
\BigZero &\widetilde{\psi_{r-1}} \\
\end{matrix}
\right)
\]
Here $\widetilde{\psi_{r-1}}$ is the map in degree 2 of the Koszul complex on $S$ \wrt \ $a_2,\ldots, a_r$. Set
$\widetilde{w_{r-1}} = [\phi(a_2),\ldots,\phi(a_r)]^{tr}$. By induction hypothesis we have that $\widetilde{\psi_{r-1}}
\widetilde{w_{r-1}} = 0$. Notice
\[
\psi_rw_r = \left(
\begin{matrix}
-a_2\phi(a_1) + a_1\phi(a_2)\\
-a_3\phi(a_1) + a_1\phi(a_3) \\
\cdots \\
-a_r\phi(a_1) +a_1\phi(a_r) \\
\widetilde{\psi_{r-1}}\widetilde{w_{r-1}} \\
\end{matrix}
\right)
\]
Since $\phi$ is $R$-linear and $\widetilde{\psi_{r-1}}\widetilde{w_{r-1}} = 0$ we get that $\psi_r w_r = 0$.
\end{proof}
As an application of our result we prove the following:
\begin{proposition}\label{app}
Let $K$ be a field of characteristic zero, $R = K[X_1,\ldots,X_n]$ and let $I, J$ be  ideals in $R$. Let $A_n(K) = K<X_1,\ldots,X_n, \partial_1, \ldots, \partial_n>$  be the $n^{th}$ Weyl algebra over $K$. Then the  Mayer-Vietoris sequence
\[
\cdots \rt H^{i}_{I+J}(R) \xrightarrow{\rho^i} H^i_I(R) \oplus H^i_J(R) \xrightarrow{\pi^i} H^i_{I\cap J}(R) \xrightarrow{\delta^i} H^{i+1}_{I+J}(R) \rt \cdots
\]
is a sequence of $A_n(K)$-modules.
\end{proposition}
\begin{proof}
Set $S = A_n(K)$.
Recall that for an $R$-module $N$ and an ideal $I$ in $R$ the $I$-torsion module of $N$ is
\[
\Gamma_I(N) = \{ m \in N \mid I^sm = 0 \ \text{for some} \ s \geq 1; \text{$s$ depending on $m$} \}.
\]

Let $E^{\bullet}$ be an injective resolution of $R$ considered as a $S$-module. Note that the following sequence  of complexes of $R$-modules
\[
0 \rightarrow \Gamma_{I+J}(E^{\bullet})  \rightarrow \Gamma_{I}(E^{\bullet}) \oplus \Gamma_{J}(E^{\bullet})   \rightarrow \Gamma_{I\cap J}(E^{\bullet}) \rightarrow 0,
\]
 is exact, see \cite[page 154]{a7}.
It can be easily verified that if $M$ is a $S$-module then $\Gamma_I(M)$ is a $S$-submodule of $M$. Thus the above sequence is an exact sequence
of a complex of $S$-modules. The result follows.
\end{proof}
\section{Indecomposable injective $S$-modules}
\s \label{As} {\textbf{Assumptions:}} In this section we assume that $S$ is a left Noetherian ring containing a regular commutative ring $R$ as a subring. Assume that $S$ considered as a right $R$-module is projective.   We make a further assumption:

$(*)$  Given an ideal $I$ in $R$ and $s \in S$, there exists $r \geq 1$ ($r$ depending on $s$) such that $I^rs \subseteq SI$.

In Theorem \ref{main} we proved that every injective left $S$-module $E$ is an injective $R$-module. In this section we investigate the structure of $E$.

\begin{remark}
The assumption $(*)$ above is trivially satisfied if $R \subseteq Z(S)$. So our examples 4,5, 8, 10, 11, 12 trivially satisfy
our hypothesis. It takes a little work to show that examples 1,2,3, 7, 9 satisfy $(*)$. We prove it for example 7. The other cases being similar.  Let $R$ be any commutative regular ring and let $\delta \colon R \rt R$
be a derivation on $R$. Consider the differential polynomial ring $S = R[X; \delta]$.
An element $s$ in $S$ has a unique expression $s = \sum_{i = 0}^{m} a_i X^i$. It suffices to take  $s = aX^m$. Note that
$Xb = bX + \delta(b)$ for any $b \in R$. We can prove that
\[
bX^n = \sum_{i =0}^{n}(-1)^i \binom{n}{i} X^{n-i}\delta^i(b).
\]
So if $b \in I^{m+1}$ then $\delta^i(b) \in I$ for all $ i = 0,\ldots, m$. Notice that
\[
bs = baX^m  = ab X^m = \sum_{i =0}^{m}(-1)^i \binom{m}{i} aX^{n-i}\delta^i(b) \in SI.
\]
Thus $I^{m+1}s \subseteq SI$.
\end{remark}

The significance of our assumption (*) is the following result.

\begin{proposition}\label{Gamma}
(with assumptions as in \ref{As}). For any ideal $I$ in $R$ and a $S$-module $M$ we have that $\Gamma_I(M)$ is a
$S$-submodule of $M$.
\end{proposition}
\begin{proof}
Clearly $\Gamma_I(M)$ is a
$R$-submodule of $M$. So let $s \in S$ and $m \in \Gamma_I(M)$. Say $I^l m = 0$. Set $J = I^l$. By our assumption
$(*)$ we have that there exists $r \geq 1$ such that $J^rs \subseteq SJ$. So $J^rsm =0$. Thus $I^{rl}sm = 0$. Therefore $sm \in \Gamma_I(M)$.
\end{proof}

\s \emph{Disussion:} Let $M$ be a $S$-module and consider the injective hull $E_S(M)$. Now $\Ass_R E_S(M)$ is a non-empty set and as $R$ is Noetherian $\Ass_R E_S(M)$ has maximal elements \wrt \ inclusion.

\begin{proposition}\label{maxA}
(with assumptions as in \ref{As}). Let $M \subseteq N$ be an essential extension of $S$-modules. Let $P \in \Ass_R N$ be a maximal element in $\Ass_R N$. Then $P \in \Ass_R M$. In particular this result holds when $N = E_S(M)$, the injective hull of $M$ as a $S$-module.
\end{proposition}
\begin{proof}
Let $P = (0 \colon x)$ for some non-zero $x \in N$. Notice $Sx \supseteq Rx \neq 0$. Since $N$ is an essential extension of $M$ we have that
$Sx \cap M \neq 0$. So there exists $s \in S$ with $sx \in M$ and $sx \neq 0$. Let
\[
\mathcal{F} = \{ (0 \colon t)  \mid t \in N, t\neq 0 \}
\]
Maximal elements in $\mathcal{F}$ are  the associate primes of $N$. Note that $(0 \colon sx) \in \mathcal{F}$. So $(0\colon sx) \subseteq Q$
for some $Q \in \Ass_R N$.

By our assumption (*) we have that $P^r s \subseteq SP$ for some $r \geq 1$. As $Px = 0$ we obtain that $P^r sx = 0$. Thus $P^r \subseteq Q$. Since $Q$ is prime we have that $P \subseteq Q$. By our choice of $P$ we have that $Q = P$.

Set $y = sx$. Thus we have $P^r \subseteq (0 \colon y) \subseteq P$. We localize at $P$. So we have
\[
P^r R_P \subseteq (0 \colon_{R_P} y) \subseteq P R_P.
\]
As $P^rR_P y = 0$ and $y \neq 0$, there exists $i \geq 1$ with $P^{i-1}R_Py \neq 0$ and $P^iR_Py = 0$. Let $u \in P^{i-1}R_Py$ be non-zero. Then $PR_Pu = 0$. Since $PR_P$
is maximal ideal in $R_P$ we get that $(0 \colon u) = PR_P$. Thus $PR_P$ is associate to $M_P$. Therefore $P$ is an associate prime of $M$.
\end{proof}

Since $S$ is left Noetherian, every injective left module over $S$ is a direct sum of indecomposable injective modules.
We prove Theorem \ref{ind-I}. For the convenience of the reader we restate it here.

\begin{theorem}\label{ind}
(with assumptions as in \ref{As}). Let $M$ be an $S$-module with $E_S(M)$ an indecomposable injective $S$-module. Then
$\Ass_R E_S(M)$ has a unique maximal element $P$. Furthermore $P \in \Ass_R M$.
\end{theorem}
\begin{proof}
Let $P, Q$ be maximal associate primes of $E_S(M)$. We prove that $P = Q$.
By \ref{maxA} $P$ is also an associate prime of $M$.
By \ref{Gamma} we have that $\Gamma_P(M)$  is a $S$-submodule of $M$.
Since $P$ is an associate prime of $M$ we have that $\Gamma_P(M) \neq 0$. Furthermore as $P$ is a maximal associate prime of $M$ it can be easily
verified that $\Ass_R \Gamma_P(M) = \{ P \}$. By \cite[2.2]{Matlis}, we have that $E_S(M) = E_S(\Gamma_P(M))$. Since $Q$ is a maximal associate prime
of $E_S(M)$, by \ref{maxA} we get $Q \in \Ass_R \Gamma_P(M) = \{ P \}$. So $Q = P$.
\end{proof}
Next we consider the injective hull of simple $S$ modules. Recall that an $S$-module $M$ is simple if the only submodules of $M$ are $0$ and $M$. It is easy to see that $M$ is simple if and only if $M \cong S/J$ where $J$ is a maximal left ideal in $S$.
\begin{proposition}
Let $J$ be a maximal left ideal in $S$. Then $I = J \cap R$ is a primary ideal in $R$. Say $I$ is $P$-primary. Then $\Ass_R S/J = \{ P \}$. Furthermore $P$ is the unique maximal element of $\Ass _R E_S(S/J)$.
\end{proposition}
\begin{proof}
Suppose $ab \in I$ and $a \notin I$. So $a \notin J$. Thus $J + Sa = S$. So $b = j + sa$ for some $j \in J$ and $s \in S$. Set $K = Rb$. Then by our assumption $(*)$, there exists
$r \geq 1$ such that $b^rs \in SK$. So $b^rs = \sum_{i = 1}^{l}s_ib = db$ for some $s_i,d \in S$. Thus $b^{r+1} = b^rj + dba$. So we get that
$b^{r+1} \in J$. Thus $b^{r+1} \in I$. Therefore  $I$ is a primary ideal in $R$.

Say $I$ is $P$-primary. So $P \in \Ass_R R/I \subseteq \Ass_R S/J$. We claim that $\Ass_R S/J = \{ P \}$. Let $Q$ be a maximal element of $\Ass_R S/J$.
Then $\Ass_R \Gamma_Q(S/J) = \{Q \}$. Furthermore by  \ref{Gamma} we have that $\Gamma_Q(S/J)$  is a $S$-submodule of $S/J$. As $S/J$ is simple and
$\Gamma_Q(S/J) \neq 0  $ we have that $\Gamma_Q(S/J) = S/J$. Thus it follows that $Q = P$ and    $\Ass_R S/J = \{ P \}$. By \ref{ind},
$P$ is the unique maximal element of $\Ass_R E_S(S/J)$.
\end{proof}

When $R $ is in the center of $S$ then we can say more.
\begin{theorem}\label{Z}
(with assumptions as in \ref{As}) Further assume that $R \subseteq Z(S)$. Let $M$ be a $S$-module. Then $\Ass_R M = \Ass_R E_S(M)$.
\end{theorem}
The main ingredient of \ref{Z} is the following:
\begin{lemma}\label{essT}
Let $R \subset S$ be an extension of rings with $S$ left Noetherian and $R$ commutative. Assume $R \subseteq Z(S)$.
Let $M, N$ be left $S$-modules with $M \subseteq N$ an essential extension of $S$-modules. Let $T$ be any multiplicatively closed subset of $R$. Then
$T^{-1}M \subseteq T^{-1}N$ is an essential extension of left $T^{-1}S$-modules.
\end{lemma}
The proof of Lemma \ref{essT} is similar to \cite[3.2.5]{BH}. The assumption $R \subseteq Z(S)$ is used to conclude that $\ann_S(x) \subseteq \ann_S(tx)$ for
any $x \in N$ and $t \in T$. We now give
\begin{proof}[Proof of Theorem \ref{Z}]
As $M \subseteq E_S(M)$ we have that $\Ass_R M \subseteq \Ass_R E_S(M)$. Conversely let $P \in \Ass_R E_S(M)$. Set $T = R\setminus P$. Consider the extension $T^{-1}R \subseteq T^{-1}S$. Notice that $T^{-1}R \subseteq Z(T^{-1}S)$. It follows that the extension $T^{-1}R \subseteq T^{-1}S$.
satisfies our assumptions \ref{As}. By Lemma \ref{essT} we get that $T^{-1}E_S(M)$ is an essential extension of $T^{-1}M$. Clearly $PT^{-1}R$ is a maximal element of $\Ass_{T^{-1}R}T^{-1}E_S(M)$. So by \ref{maxA} we get that $PT^{-1}R \in \Ass_{T^{-1}R} M$. Thus $P \in \Ass_R M$.
\end{proof}
\section{The case when $S$ is commutative and a finite extension of $R$}
In this section  $R, S$ are commutative and $S$ a finite extension of $R$ and a projective $R$-module. Recall that we assume $R$ to be regular.
It is easily seen that   $S$ is a \CM \ ring. In-fact $S$ is a projective $R$-module if and only if $S$ is \CM.

We now give a proof of Theorem \ref{CM-I}. For convenience of the reader we restate it here. 
\begin{theorem}\label{CM}
(with assumptions as above). Let $\m$ be a maximal ideal of $S$. Set $\n = \m \cap R$.
Let $\n S = Q_1\cap Q_2\cap\cdots \cap Q_c$ be a minimal primary decomposition of $\n S$ where $Q_1$ is $\m$-primary.
Then
\[
E_S\left(\frac{S}{\m}\right) = E_R\left(\frac{R}{\n}\right)^{\ell_R(S/Q_1)}.
\]
\end{theorem}
\begin{proof}
Notice that $\n$ is a maximal ideal of $R$. Set $E = E_S(S/\m)$. By Theorem 1 we have that $E$ is an injective $R$-module.

 Note that we have an injection $R/\n \hookrightarrow S/\m $ of $R$-modules and an injection $S/\m \hookrightarrow E$ of $S$-modules. Composing we have an
an injection $R/\n \hookrightarrow E$ of $R$-modules. Thus $\n \in \Ass_R E$.

\emph{Claim 1:} $\Ass_R E = \{ \n \}.$

Suppose $P \in \Ass_R E$. Say $P = (0 \colon x)$ for some non-zero $x \in E$. Each element of $E$ is annhilated by a power of $\m$.
Say $\m^t x = 0$. Then $\n^t x = 0$. This implies $\n^t \subseteq P$. As $P$ is prime we obtain $\n \subseteq P$. So $P = \n$. This proves Claim 1.

Thus by the structure theorem of injectives \cite[3.2.8]{BH} we have that
\[
E = E_R\left(\frac{R}{\n}\right)^{c} \quad \text{where} \ c = \dim_{R/\n} \Hom_R\left(\frac{R}{\n}, E \right).
\]
Notice that
\[
\soc_{R,\n} E = \{ t \in E \mid \n t = 0 \} \cong \Hom_R\left(\frac{R}{\n}, E\right).
\]

\emph{Claim 2:} $\soc_{R,\n} E = (0 \colon_E Q_1)$.

Let $t \in \soc_{R,\n} E$. Then note that $\n S t = 0$. Notice $Q_2\cap \cdots \cap Q_c \nsubseteq \m$. Say $\xi \in Q_2\cap \cdots \cap Q_c \setminus \m$. Let $a \in Q_1.$ So $a\xi \in \n S$. Thus $a\xi t = 0$. Now $\Ass_S E = \{ \m\}$. This implies that $\xi$ is a non-zero divisor on $E$. So
$at = 0$. Thus $Q_1 t = 0$.

Conversely if $t \in (0 \colon_E Q_1)$, then $Q_1t = 0$. Since $\n \subseteq \n S \subseteq Q_1$ we get that $\n t = 0$. Thus Claim 2 is proved.

Notice $(0 \colon_E Q_1) \cong \Hom_S(S/Q_1, E)$. Since $\Supp_S(S/Q_1) = \{ \m \}$ and as $E = E_S(S/\m)$ we get that
$\ell_S(\Hom_S(S/Q_1, E)) = \ell_S(S/Q_1)$.

Notice that since $\n S \subseteq Q_1 \subseteq \m$ we have that $Q_1 \cap R = \n$. So $S/Q_1$ is a finite dimensional vector space over $R/\n$.
Again as $\Supp_S(S/Q_1) = \{ \m \}$ we get that $\ell_R (S/Q) = r \ell_S (S/Q_1)$ where $r = \dim_{R/\n} S/\m$.

Thus
\[
c = \dim_{R/\n}\soc_{R,\n} E = \dim_{R/\n} (0 \colon_E Q_1) = r \ell_S (S/Q_1) = \ell_R(S/Q_1).
\]
\end{proof}


\begin{thebibliography}{10}

\bibitem{Bhat} S.~M.~Bhatwadekar,
Cancellation theorems for projective modules over a two-dimensional ring and its polynomial extensions.
Compositio Math. 128 (2001), no. 3, 339–-359.

\bibitem {B} J.-E. Bj{\"{o}}rk,
Rings of differential operators.
North-Holland Mathematical Library, 21. North-Holland Publishing Co., Amsterdam-New York, 1979.


\bibitem {BH}  W. Bruns and J. Herzog,
Cohen-Macaulay Rings, revised edition,
Cambridge Studies in Advanced Mathematics, 39.
Cambridge University Press, 1998.



\bibitem{E} D.~Eisenbud,
Commutative algebra; With a view toward algebraic geometry.
 Graduate Texts in Mathematics, 150. Springer-Verlag, New York, 1995.

\bibitem{a7} S.~B.~Iyengar; G.~J.~Leuschke; A.~Leykin; C.~Miller; E.~Miller; A.~K.~Singh and U.~Walther,
Twenty-four hours of local cohomology.
Graduate Studies in Mathematics, 87. American Mathematical Society, Providence, RI, 2007.


\bibitem {Ly} G. Lyubeznik,
Finiteness properties of local cohomology modules (an application of D-modules to commutative algebra).
Invent. Math. 113 (1993), no. 1, 41–-55.

\bibitem{Mann} H.~B.~Mann,
 On integral bases.
 Proc. Amer. Math. Soc. 9 (1958), 167–-172.

\bibitem{Matlis} E.~Matlis,
Injective modules over Noetherian rings.
Pacific J. Math. 8 (1958), 511--528.

\bibitem {M} H.~Matsumura,
 Commutative ring theory. Translated from the Japanese by M. Reid. Second edition.
 Cambridge Studies in Advanced Mathematics, 8. Cambridge University Press, Cambridge, 1989

\bibitem {MS} R.~MacKenzie and J.~Scheuneman,
A number field without a relative integral basis.
Amer. Math. Monthly. 78 (1971), 882–-883.

\bibitem {P} T.~J.~Puthenpurakal,
Koszul homology of local cohomology modules,
\textit{Preprint}.

\end{thebibliography}
\end{document}